\documentclass[12pt]{amsart}
\usepackage{amsmath,amssymb,hyperref,geometry,graphicx,amsthm,tikz,enumitem,mathtools,textcomp}
\usepackage[all,cmtip]{xy}

\newtheorem{dummy}{}[]
\newtheorem{theorem}[dummy]{Theorem}

\newtheorem{proposition}[dummy]{Proposition}
\theoremstyle{definition}

\newtheorem{remark}[dummy]{Remark}

\title{The Two-Step Property and the Mathematics of Musical Scale Size}

\author[E.~Clader]{Emily Clader}
\address{Department of Mathematics, San Francisco State University, United States of America}
\email{eclader@sfsu.edu}

\author[V.~Jelmyer]{Vanessa Jelmyer}
\address{Department of Mathematics, San Francisco State University, United States of America}
\email{vjelmyer@mail.sfsu.edu}


\begin{document}

\begin{abstract}
A Pythagorean scale is a mathematical encoding of a musical scale as a finite list of numbers of the form $3^b/2^a$.  Previous work of the first author discussed the 2-step property as a way to measure which Pythagorean scales are the most ``evenly-spaced.''  In this paper, we give a mathematician's account of the characterization of the Pythagorean scales that have the 2-step property; compellingly, the list includes the 5-note, 7-note, and 12-note Pythagorean scales, which are well-known as the pentatonic, diatonic, and chromatic scales of music theory.  (After this preprint was initially posted, it was brought to our attention that these results are not new: they have previously appeared in the work of Carey and Clampitt, and related work has been done by several other members of the music theory community, now cited below.  We leave this preprint available because it is written for mathematicians with no musical background, and as such may provide helpful exposition for some audiences, but we no longer make any claim to originality of the content.)
\end{abstract}

\maketitle

\section{Introduction}

Musical notes can be encoded mathematically by identifying them with positive real numbers that represent their frequencies.  A musical scale---at least for the purposes of this work---consists of a finite collection of notes between a given starting note and the note one octave above.  Since notes that differ by an octave have frequencies that differ by a factor of two, we can view a scale mathematically (after normalizing) as a finite collection of numbers in the interval $[1,2]$ that includes the two endpoints.

One particularly classical way to produce scales in this sense is to consider the first $n-1$ powers of $3$ for some integer $n \geq 1$, and to divide each by an appropriate power of $2$ to ensure that the result lies in the interval $[1,2]$.  This is referred to as the {\bf $n$-note Pythagorean scale}.\footnote{For readers with more musical background, it is worth noting that powers of 3 correspond to intervals of a fifth, musically, so the Pythagorean method of producing scales is rooted in the musical notion of the ``circle of fifths."}  For instance, when $n=5$, the resulting scale (rearranged so that its elements are in increasing numerical order) is the following:
\[\frac{3^0}{2^0}, \; \frac{3^2}{2^3}, \; \frac{3^4}{2^6}, \; \frac{3^1}{2^1}, \; \frac{3^3}{2^4}, \; \frac{3^0}{2^{-1}}.\]

When a scale is arranged in increasing numerical order in this way---meaning its notes are ``ascending'' in the musical sense---the ratios between successive elements determine the musical distance between consecutive notes.  It might be desirable for a scale to consist of $n$ perfectly evenly-spaced notes, so that a melody played in that scale could be shifted up or down by some number of scale steps (``transposed,'' musically) without changing its essential character.  However, the only way that an $n$-note scale can be perfectly evenly-spaced in this way is if it is
\[1, \; 2^{1/n}, \; 2^{2/n}, \ldots, \; 2^{(n-1)/n}, \; 2,\]
so in particular, no Pythagorean scale can be perfectly evenly-spaced.

Perhaps the next best thing that one might hope for is that the scale has exactly two distinct ratios between successive elements when arranged in increasing numerical order; following previous work \cite{Clader} of the first author, we say that such a scale has the {\bf 2-step property}.  The main theorem of \cite{Clader} showed that the $n$-note Pythagorean scale has the 2-step property whenever $n$ is the denominator of a convergent---that is, a truncation of the continued fraction---of $\log_2(3)$.  In this work, we show that, by expanding attention from convergents to semi-convergents, one obtains a complete characterization of the Pythagorean scales with the 2-step property:

\begin{theorem}
\label{thm:main}
Let $n \geq 2$ be an integer.  Then the $n$-note Pythagorean scale has the $2$-step property if and only if $n$ is the denominator of a semi-convergent of $\log_2(3)$.
\end{theorem}

See Section~\ref{sec:background} for a review of the notions of convergents and semi-convergents, and in particular Remark~\ref{rmk:whyCF} for a heuristic explanation of their appearance in this context.

\begin{remark}
\label{remark}
After this preprint was initially posted, it was brought to our attention that the result of Theorem~\ref{thm:main} is not new: it is stated as the ``Characterization Theorem'' in the paper \cite{CC1989} of Carey and Clampitt and is proven in \cite{CC2012}.  Carey and Clampitt's work fits into a large body of research in the music theory community on this subject (of which we were unaware at the time of our writing, and which we are grateful to an anonymous referee for pointing out), including Wilson's work on moment of symmetry (MOS) scales \cite{Wilson}, Hellegouarch's work on natural scales \cite{Hellegouarch}, and Cubarsi's work on cyclic scales \cite{Cubarsi}.  We leave this preprint available because it is written for mathematicians with no musical background, and as such may provide helpful exposition for some audiences, but we no longer make any claim to originality of the content.
\end{remark}

Several of the first few semi-convergents of $\log_2(3)$ are
\[\frac{3}{2}, \; \frac{8}{5}, \; \frac{5}{3}, \; \frac{19}{12}, \; \frac{11}{7}, \ldots\]
Compellingly, the Pythagorean scales with sizes given by the denominators of these semi-convergents include the 5-note scale, the 7-note scale, and the 12-note scale.  All three of these are very important musically: the 5-note scale is the ``pentatonic'' scale that appears prominently in the music of China, Scotland, and other traditions, as well as in certain types of jazz; the 7-note scale is the ``diatonic'' or ``do-re-mi'' scale that can be played, for example, on the white keys of a piano; and the 12-note scale is the ``chromatic'' scale that can be played by using all the keys, black and white, on a piano.

\begin{remark}
It is worth acknowledging that, both musically and mathematically, the 7-note scale and the 12-note scale are quite different in character.  Musically, the chromatic scale is understood to be perfectly evenly-spaced, so the 12-note Pythagorean scale can be viewed as an imperfect approximation of an evenly-spaced scale, and the fact that it is a good approximation is evidenced by the fact that (a) there are only two distinct step sizes, as confirmed by Theorem~\ref{thm:main}; and (b) the two step sizes are quite close to one another.  In particular, in the notation introduced in Section~\ref{sec:background}, we will see that the two step sizes in the 12-note scale are
\[I_4^{-1} = \frac{3^{-5}}{2^{-8}} \approx 1.0535 \;\;\;\; \text{ and } J_4^{-1} = \frac{3^7}{2^{11}} \approx 1.0679.\]
By contrast, the diatonic scale is understood musically as a subset of the chromatic scale, where some of its notes are separated by a single chromatic scale-step and others are separated by two chromatic scale-steps.  Approximating this scale by the 7-note Pythagorean scale, we indeed see two distinct step sizes arising (as confirmed, again, by Theorem~\ref{thm:main}), and as we will see in Section~\ref{sec:proof1}, these two step sizes are
\[I_4^{-1} \approx 1.0535 \;\;\;\; \text{ and } J_4^{-1}I_4^{-1} = 1.125.\]
These are much less close numerically than the two step sizes of the 12-note Pythagorean scale, but this makes perfect sense musically: while the 2-step property of the 12-note Pythagorean scale is an approximation of the impossibility of precisely even spacing, the 2-step property of the 7-note Pythagorean scale is a manifestation of approximating a scale that was intended to have two step sizes in the first place.
\end{remark}

\subsection*{Acknowledgments}  The authors thank Matthias Beck, ChunKit Lai, and Kim Seashore for valuable conversations related to this project.  We are particularly grateful to the anonymous referee who pointed us to the music-theoretic literature on this topic, which shows that these results have been proven and studied by numerous others.  The first author was supported by NSF CAREER grant 2137060.

\section{Background}
\label{sec:background}

We begin with some basic background on continued fractions before turning to the more specific background on Pythagorean scales and the $2$-step property.

\subsection{Continued fractions}

A {\bf finite continued fraction}\footnote{The continued fractions considered in this paper are sometimes referred to as ``simple'' continued fractions, which are a special case of a more general notion obtained by replacing the ones in the numerators with other integers.} is an expression of the form
\[[k_0; k_1, \ldots, k_i] \coloneq k_0+\frac{1}{k_1+\frac{1}{\displaystyle k_2+\cdots+ \frac{1}{k_i}}}\]
where $k_0, \ldots, k_i$ are positive integers.  Such an expression is manifestly a rational number.  A {\bf continued fraction} is a limit of such expressions:
\[[k_0; k_1, k_2, \ldots] \coloneq k_0+\frac{1}{k_1+\frac{1}{\displaystyle k_2+\cdots}} = \lim_{i \rightarrow \infty} [k_0; k_1, \ldots, k_i]\]
for a sequence of integers $\{k_i\}$.  It is well-known (see, for example, \cite[Theorem 14]{Khinchin}) that every irrational real number has a unique expression as a continued fraction, whose terms $k_i$ can be calculated explicitly via a version of the Euclidean algorithm.  The finite continued fractions $[k_0; k_1, \ldots, k_i]$ obtained by truncating an (infinite) continued fraction are referred to as its {\bf convergents}, and if
\[x = [k_0; k_1, k_2, \ldots]\]
for some irrational number $x$, these convergents are rational approximations of $x$.  We denote the numerator and denominator of the convergents by
\[[k_0; k_1, \ldots, k_i] =: \frac{a_i}{b_i}.\]
These satisfy a fundamental and well-known recursion (see, for example, \cite[Theorem 1]{Khinchin}):
\begin{align}
\label{eq:recursion}
    \nonumber &a_i = k_ia_{i-1} + a_{i-2},\\
    &b_i = k_ib_{i-1} + b_{i-2},
\end{align}
with base cases given by
\begin{align*}
    &a_{-1} = 1, \qquad a_0 = k_0,\\
    &b_{-1} = 0, \qquad b_0 = 1.
\end{align*}

Interpolating between the convergents are the {\bf semi-convergents}, which are the rational numbers
\begin{equation}
    \label{eq:semiconvergents}
    \frac{a_{i-1} + ka_{i}}{b_{i-1} + kb_{i}}
\end{equation}
for $0 \leq k \leq k_{i+1}$ and $i \geq 1$.  Note that the cases $k=0$ and $k=k_{i+1}$ recover the convergents as special cases of semi-convergents, in light of \eqref{eq:recursion}.  Both the convergents and semi-convergents of $x$ are ``best approximations'' of $x$ in the sense that they are closer to $x$ than any other rational number with smaller or equal denominator \cite[Theorem 15]{Khinchin}.

Throughout this paper, we will always consider the continued fraction expression for the irrational number
\[\log_2(3) = 1.5849625\ldots,\]
which begins
\[\log_2(3) = [1; 1, 1, 2, 2, 3, 1, 5, 2, 23, \ldots].\]
In particular, the first few convergents are as follows:
\begin{align}
\label{eq:convergents}
\nonumber &[1;1] = 1+ \frac{1}{1} = \frac{2}{1}\\
\nonumber &[1;1,1] = 1+ \frac{1}{1+\frac{1}{1}} = \frac{3}{2} = 1.5\\
\nonumber &[1;1,1,2] = 1+ \frac{1}{1+\frac{1}{1+ \frac{1}{2}}} = \frac{8}{5} = 1.6\\
&[1;1,1,2,2] = 1+ \frac{1}{1+\frac{1}{1+ \frac{1}{2+ \frac{1}{2}}}} = \frac{19}{12} = 1.58333\ldots.
\end{align}
From these calculations, we see that the first few values of $a_i$, $b_i$, and $k_i$ are given by the following table, from which the first few semi-convergents are easily read:
\begin{equation}
\label{eq:tableofks}
\begin{tabular}{c|c|c|c}
$i$ & $a_i$ & $b_i$ & $k_i$\\
\hline
0 & 1 & 1 & 1\\
1 & 2 & 1 & 1\\
2 & 3 & 2 & 1\\
3 & 8 & 5 & 2\\
4 & 19 & 12 & 2.
\end{tabular}
\end{equation}
For instance, the semi-convergents with $i=2$ have $0 \leq k \leq k_3=2$, so they are
\[\frac{a_1}{b_1} = \frac{2}{1}, \;\;\; \frac{a_1+a_2}{b_1+b_2} = \frac{5}{3}, \;\;\; \frac{a_1+2a_2}{b_1+2b_2} = \frac{3}{2}.\]
The first and last of these are consecutive convergents of $\log_2(3)$, whereas the second interpolates between those convergents.  Along the same lines, the list of the first few semi-convergents of $\log_2(3)$ is as follows (where the boxed entries are the convergents):
\[\boxed{\frac{2}{1}}, \; \frac{5}{3}, \; \boxed{\frac{3}{2}}, \; \frac{11}{7}, \; \boxed{\frac{8}{5}}, \; \frac{27}{17}, \; \frac{46}{29}, \boxed{\frac{19}{12}}, \; \ldots\]
The denominators of these semi-convergents will be the scale sizes of interest, so we now turn to a discussion of Pythagorean scales.

\subsection{Pythagorean scales with the 2-step property}

Recall from the introduction that the {\bf $n$-note Pythagorean scale} is the list of numbers
\[P_n = \left\{ \frac{3^b}{2^a} \; \left| \; 0 \leq b < n, \;\; 1 \leq \frac{3^b}{2^a} \leq 2\right.\right\}, \]
where $a$ and $b$ are integers.  It is straightforward to see that, whereas both $\frac{3^0}{2^0} =1$ and $\frac{3^0}{2^{-1}} = 2$ are elements of $P_n$, as long as $0 < b < n$ there is a unique integer $a$ such that $\frac{3^b}{2^a}$ is in $P_n$.

For example, the 5-note Pythagorean scale has the following elements (which we refer to, in musical parlance, as its ``notes''):
\[P_5 = \left\{\frac{3^0}{2^0}, \; \frac{3^0}{2^{-1}}, \; \frac{3^1}{2^1}, \; \frac{3^2}{2^3}, \; \frac{3^3}{2^4}, \; \frac{3^4}{2^6}\right\}.\]
Arranged in increasing numerical order, these are as follows:
\begin{equation}
\label{eq:P5}
P_5 = \left\{\frac{3^0}{2^0}, \; \frac{3^2}{2^3}, \; \frac{3^4}{2^6}, \; \frac{3^1}{2^1}, \; \frac{3^3}{2^4}, \; \frac{3^0}{2^{-1}}\right\}.
\end{equation}

We say that a Pythagorean scale has the {\bf $k$-step property} if there are exactly $k$ distinct ratios between successive notes in the scale, when arranged in increasing numerical order.  For example, $P_5$ has the 2-step property, since the ratio between the successive elements in \eqref{eq:P5} are all either
\[\frac{3^2}{2^3} \;\;\; \text{ or }  \;\;\; \frac{3^{-3}}{2^{-5}}.\]
By contrast, the $6$-note Pythagorean scale has the 3-step property because, arranged in increasing numerical order, we have
\[P_6 = \left\{\frac{3^0}{2^0}, \; \frac{3^2}{2^3}, \; \frac{3^4}{2^6}, \; \frac{3^1}{2^1}, \; \frac{3^3}{2^4}, \; \frac{3^5}{2^7}, \; \frac{3^0}{2^{-1}}\right\},\]
which has the three distinct ratios
\[\frac{3^2}{2^0}, \;\;\;\;\;\; \frac{3^{-3}}{2^{-5}}, \;\;\; \text{ and } \;\;\; \frac{3^{-5}}{2^{-8}}\]
between successive notes.

Roughly speaking, the smaller the $k$ for which $P_n$ has the $k$-step property, the more ``evenly-spaced'' the notes in $P_n$ are.  As observed in the introduction, no Pythagorean scale can have the $1$-step property, so we are primarily interested in determining which Pythagorean scales have the $2$-step property, as these are the most evenly-spaced.

A large family of Pythagorean scales with the $2$-step property was described in previous work of the first author \cite{Clader} (which, as noted in Remark~\ref{remark}, unknowingly replicated work of Carey and Clampitt \cite{CC1989, CC2012}).  Namely, if $a_i$ and $b_i$ are the numerator and denominator of the $i$th convergent of $\log_2(3)$, as above, and we set
\[I_i:= \frac{3^{b_{i-1}}}{2^{a_{i-1}}}, \;\;\;\; J_i:= \frac{3^{b_{i-1}-b_i}}{2^{a_{i-1}-a_i}},\]
then the result is the following:

\begin{theorem}[\cite{Clader, CC1989, CC2012}]
\label{thm:old}
For any integer $i \geq 1$, the $b_i$-note Pythagorean scale has the 2-step property.  More specifically, the $b_i$-note Pythagorean scale has one of the following two types:
\begin{itemize}
\item {\bf Type A}: The scale consists of $b_{i-1}$ blocks, each of which has either $k_i$ or $k_i-1$ steps of size $I_i$ followed by one step of size $J_i$; or
\item {\bf Type B}: The scale consists of $b_{i-1}$ blocks, each of which has one step of size $J_i^{-1}$ followed by either $k_i$ or $k_i-1$ steps of size $I_i^{-1}$.
\end{itemize}
\end{theorem}

For example, the table in \eqref{eq:tableofks} shows that $a_2 = 3$, $b_2 = 2$, $a_3 = 8$, and $b_3 = 5$, so
\[I_3 = \frac{3^2}{2^3} \;\;\; \text{ and } \;\;\; J_3 = \frac{3^{-3}}{2^{-5}}.\]
The Pythagorean scale with $b_3=5$ notes, which we considered above, indeed can be broken down into $b_2 = 2$ blocks (boxed in the diagram below), the first of which has $k_3 = 2$ steps of size $I_3$ followed by one step of size $J_3$ and the second of which has $k_3 - 1 = 1$ step of size $I_3$ followed by one step of size $J_3$:

\begin{equation}
\label{eq:P5steps}
    \begin{tikzpicture}
        % make dots big
        \tikzstyle{dot}=[circle, fill, inner sep=0pt, minimum size=0.3cm]

        % Left box
        \draw[thick] (0,-0.7) -- (3,-0.7);
        \draw[thick] (3,-0.7) -- (3, -0.5);
        \draw[thick] (3,0.5) -- (3,1.75);
        \draw[thick] (3,1.75) -- (0,1.75);
        \draw[thick] (0,1.75) -- (0,0.5);
        \draw[thick] (0,-0.5) -- (0,-0.7);

        % Draw notes
        \node (A1) at (0,0) {$\displaystyle\frac{3^0}{2^0}$};
        \node (A2) at (1,0) {$\displaystyle\frac{3^2}{2^3}$};
        \node (A3) at (2,0) {$\displaystyle\frac{3^4}{2^6}$};
        \node (A4) at (3,0) {$\displaystyle\frac{3^1}{2^1}$};

        % Draw curvy arrows (clockwise from pi to 0)
        \draw[->, thick, >=stealth, bend left=75] (A1) to node[above] {\( I_3 \)} (A2);
        \draw[->, thick, >=stealth, bend left=75] (A2) to node[above] {\( I_3 \)} (A3);
        \draw[->, thick, >=stealth, bend left=75] (A3) to node[above] {\(J_3\)} (A4);

        % right diagram . shifted right 3.8cm
        \begin{scope}[shift={(3,0)}]
            \draw[thick] (0,-0.7) -- (2,-0.7);
            \draw[thick] (2,0.5) -- (2, 1.75);
            \draw[thick] (2,-0.7) -- (2, -0.5);
            \draw[thick] (2,1.75) -- (0,1.75);

        % Draw notes
        \node (B2) at (1,0) {$\displaystyle\frac{3^3}{2^4}$};
        \node (B3) at (2,0) {$\displaystyle\frac{3^0}{2^{-1}}$};

         % Draw curvy arrows 
         \draw[->, thick, >=stealth, bend left=75] (A4) to node[above] {\( I_3 \)} (B2);
        \draw[->, thick, >=stealth, bend left=75] (B2) to node[above] {\( J_3 \)} (B3);

        \end{scope}

    \end{tikzpicture}
\end{equation}

Thus, the $5$-note Pythagorean scale has Type A.  A similar calculation applies to the Pythagorean scale with $b_4 = 12$ notes, which has Type B:

\begin{equation}
\label{eq:P12steps}
    \begin{tikzpicture}
        % make dots big
        \tikzstyle{dot}=[circle, fill, inner sep=0pt, minimum size=0.3cm]

        % Left box
        \draw[thick] (0,-0.7) -- (2,-0.7);
        \draw[thick] (2,-0.7) -- (2, -0.5);
        \draw[thick] (2,0.5) -- (2,1.75);
        \draw[thick] (2,1.75) -- (0,1.75);
        \draw[thick] (0,1.75) -- (0,0.5);
        \draw[thick] (0,-0.5) -- (0,-0.7);

        % Draw notes
        \node (A1) at (0,0) {$\displaystyle\frac{3^0}{2^0}$};
        \node (A2) at (1,0) {$\displaystyle\frac{3^7}{2^{11}}$};
        \node (A3) at (2,0) {$\displaystyle\frac{3^2}{2^3}$};

        % Draw curvy arrows (clockwise from pi to 0)
        \draw[->, thick, >=stealth, bend left=75] (A1) to node[above] {\( J_4^{-1} \)} (A2);
        \draw[->, thick, >=stealth, bend left=75] (A2) to node[above] {\( I_4^{-1} \)} (A3);

        \begin{scope}[shift={(2,0)}]
            \draw[thick] (0,-0.7) -- (2,-0.7);
            \draw[thick] (2,0.5) -- (2, 1.75);
            \draw[thick] (2,-0.7) -- (2, -0.5);
            \draw[thick] (2,1.75) -- (0,1.75);

        % Draw notes
        \node (B1) at (1,0) {$\displaystyle\frac{3^9}{2^{14}}$};
        \node (B2) at (2,0) {$\displaystyle\frac{3^4}{2^{6}}$};

         % Draw curvy arrows 
         \draw[->, thick, >=stealth, bend left=75] (A3) to node[above] {\( J_4^{-1} \)} (B1);
        \draw[->, thick, >=stealth, bend left=75] (B1) to node[above] {\( I_4^{-1} \)} (B2);

        \end{scope}

        \begin{scope}[shift={(4,0)}]
            \draw[thick] (0,-0.7) -- (3,-0.7);
            \draw[thick] (3,0.5) -- (3, 1.75);
            \draw[thick] (3,-0.7) -- (3, -0.5);
            \draw[thick] (3,1.75) -- (0,1.75);

        % Draw notes
        \node (C1) at (1,0) {$\displaystyle\frac{3^{11}}{2^{17}}$};
        \node (C2) at (2,0) {$\displaystyle\frac{3^6}{2^{9}}$};
        \node (C3) at (3,0) {$\displaystyle\frac{3^1}{2^{1}}$};

         % Draw curvy arrows 
        \draw[->, thick, >=stealth, bend left=75] (B2) to node[above] {\( J_4^{-1} \)} (C1);
        \draw[->, thick, >=stealth, bend left=75] (C1) to node[above] {\( I_4^{-1} \)} (C2);
        \draw[->, thick, >=stealth, bend left=75] (C2) to node[above] {\( I_4^{-1} \)} (C3);

        \end{scope}

        \begin{scope}[shift={(7,0)}]
            \draw[thick] (0,-0.7) -- (2,-0.7);
            \draw[thick] (2,0.5) -- (2, 1.75);
            \draw[thick] (2,-0.7) -- (2, -0.5);
            \draw[thick] (2,1.75) -- (0,1.75);

        % Draw notes
        \node (D1) at (1,0) {$\displaystyle\frac{3^8}{2^{12}}$};
        \node (D2) at (2,0) {$\displaystyle\frac{3^3}{2^{4}}$};

         % Draw curvy arrows 
         \draw[->, thick, >=stealth, bend left=75] (C3) to node[above] {\( J_4^{-1} \)} (D1);
        \draw[->, thick, >=stealth, bend left=75] (D1) to node[above] {\( I_4^{-1} \)} (D2);

        \end{scope}

        \begin{scope}[shift={(9,0)}]
            \draw[thick] (0,-0.7) -- (3,-0.7);
            \draw[thick] (3,0.5) -- (3, 1.75);
            \draw[thick] (3,-0.7) -- (3, -0.5);
            \draw[thick] (3,1.75) -- (0,1.75);

        % Draw notes
        \node (E1) at (1,0) {$\displaystyle\frac{3^{10}}{2^{15}}$};
        \node (E2) at (2,0) {$\displaystyle\frac{3^5}{2^{7}}$};
        \node (E3) at (3,0) {$\displaystyle\frac{3^0}{2^{-1}}$};

         % Draw curvy arrows 
        \draw[->, thick, >=stealth, bend left=75] (D2) to node[above] {\( J_4^{-1} \)} (E1);
        \draw[->, thick, >=stealth, bend left=75] (E1) to node[above] {\( I_4^{-1} \)} (E2);
        \draw[->, thick, >=stealth, bend left=75] (E2) to node[above] {\( I_4^{-1} \)} (E3);

        \end{scope}        

    \end{tikzpicture}
\end{equation}

Note that the boundaries between the blocks of the $b_i$-note Pythagorean scale are precisely the notes of the $b_{i-1}$-note Pythagorean scale.  This can be readily proven from the recursions \ref{eq:recursion}, and it is also the key to the proof of Theorem~\ref{thm:old}.

\begin{remark}
\label{rmk:whyCF}
Heuristically, one might expect that the Pythagorean scales whose size is the denominator of a rational approximation of $\log_2(3)$ are among the most evenly-spaced.  Indeed, the approximate equality
\[\log_2(3) \approx \frac{a_i}{b_i}\]
rearranges to
\[3^{b_i} \approx 2^{a_i}.\]
This means that, as one adds the notes with numerator $3^0$, $3^1$, and so on to a Pythagorean scale, then after adding the note with numerator $3^{b_{i}-1}$ (which is the last note added to the $b_i$-note Pythagorean scale), the next note would be either 
\[\frac{3^{b_i}}{2^{a_i}} \approx 1 \;\;\; \text{ or } \;\;\;  \frac{3^b}{2^{a_i -1}} \approx 2,\]
depending on whether $3^{b_i}> 2^{a_i}$ or $3^{b_i}<2^{a_i}$.  In this sense, the $b_i$-note Pythagorean scale terminates when adding an additional note approximately returns (up to a factor of 2) to the starting note of the scale, in much the same way that the unique $n$-note scale with the $1$-step property,
\[3^0, \; 3^{1/n}, \; 3^{2/n}, \; \ldots, 3^{n-1/n}, \; 2,\]
terminates when it returns (up to a factor of 2) to its starting note.
\end{remark}

In light of this heuristic, it is reasonable to expect that the denominators not only of convergents of $\log_2(3)$, but also of semi-convergents (which are also rational approximations of $\log_2(3)$), give rise to Pythagorean scales with the 2-step property.  This is indeed the case, and to prove it, we turn in the next section to a method for deleting notes from the $b_i$-note scale to produce a smaller Pythagorean scale that nevertheless retains the 2-step property.

\section{The Deletion Procedure}
\label{sec:proof1}

To motivate the deletion procedure, we note that, if the $b_i$-note Pythagorean scale has Type A according to Theorem~\ref{thm:old}, then deleting each of the notes between a step of size $I_i$ and a step of size $J_i$ will produce a scale with the 2-step property.  For instance, applying this idea to the 5-note Pythagorean scale \eqref{eq:P5steps} produces a scale with step sizes $I_3$ and $I_3J_3$:

\begin{equation}
\label{eq:deletionP5toP3}
    \begin{tikzpicture}
        % make dots big
        \tikzstyle{dot}=[circle, fill, inner sep=0pt, minimum size=0.3cm]

        % Left box
        \draw[thick] (0,-0.7) -- (3,-0.7);
        \draw[thick] (3,-0.7) -- (3, -0.5);
        \draw[thick] (3,0.5) -- (3,1.75);
        \draw[thick] (3,1.75) -- (0,1.75);
        \draw[thick] (0,1.75) -- (0,0.5);
        \draw[thick] (0,-0.5) -- (0,-0.7);

        % Draw notes
        \node (A1) at (0,0) {$\displaystyle\frac{3^0}{2^0}$};
        \node (A2) at (1,0) {$\displaystyle\frac{3^2}{2^3}$};
        \node (A3) at (2,0) {$\displaystyle\frac{3^4}{2^6}$};
        \node (A4) at (3,0) {$\displaystyle\frac{3^1}{2^1}$};

        \draw[thick] (1.7, -0.5) -- (2.3, 0.5);
        \draw[thick] (2.3, -0.5) -- (1.7, 0.5);

        % Draw curvy arrows (clockwise from pi to 0)
        \draw[->, thick, >=stealth, bend left=75] (A1) to node[above] {\( I_3 \)} (A2);
        \draw[->, thick, >=stealth, bend left=75] (A2) to node[above] {\( I_3J_3 \)} (A4);

        % right diagram . shifted right 3.8cm
        \begin{scope}[shift={(3,0)}]
            \draw[thick] (0,-0.7) -- (2,-0.7);
            \draw[thick] (2,0.5) -- (2, 1.75);
            \draw[thick] (2,-0.7) -- (2, -0.5);
            \draw[thick] (2,1.75) -- (0,1.75);

        % Draw notes
        \node (B2) at (1,0) {$\displaystyle\frac{3^3}{2^4}$};
        \node (B3) at (2,0) {$\displaystyle\frac{3^0}{2^{-1}}$};

        \draw[thick] (0.7, -0.5) -- (1.3, 0.5);
        \draw[thick] (1.3, -0.5) -- (0.7, 0.5);

         % Draw curvy arrows 
         \draw[->, thick, >=stealth, bend left=75] (A4) to node[above] {\( I_3J_3 \)} (B3);

        \end{scope}

    \end{tikzpicture}
\end{equation}

Somewhat surprisingly, the result is not only a list of numbers with the 2-step property, but it is a Pythagorean scale---in this case, the 3-note Pythagorean scale---because its numerators are precisely $3^0$, $3^1$, and $3^2$.

The proof of the following theorem verifies that this phenomenon holds in general.  Furthermore, one need not combine just a single step of size $I_i$ with the subsequent $J_i$ as above, but one can combine as many as $k_i -1$ steps of size $I_i$ with the subsequent $J_i$ (in Type A, and similarly in Type B).  The precise statement is the following.

\begin{proposition}
\label{prop:deletion}
For any integers $i \geq 1$ and $0 \leq k \leq k_i-1$, the $(b_i - kb_{i-1})$-note Pythagorean scale has the $2$-step property.
\end{proposition}
\begin{proof}
The case $k=0$ is the content of Theorem~\ref{thm:old}, so we assume that $1 \leq k \leq k_i - 1$.

Suppose, first, that the $b_i$-note Pythagorean scale has Type A, so that its step sizes are $I_i$ and $J_i$ according to Theorem~\ref{thm:old}.  By definition, the $(b_i - kb_{i-1})$-note is obtained from the $b_i$-note Pythagorean scale by deleting the notes $\frac{3^b}{2^a}$ with $b \geq b_i - kb_{i-1}$.   We prove, by induction on $k$, that each of these deleted notes is among the $k$ notes preceding a step of size $J_i$ in the $b_i$-note Pythagorean scale.

As a base case, let $k=1$, so the claim is that whenever $b \geq b_i - b_{i-1}$, the note $\frac{3^b}{2^a}$ in the $b_i$-note Pythagorean scale precedes a step of size $J_i$.  If not, then this note precedes a step of size $I_i$, but this would mean that the following note is
\[\frac{3^b}{2^a} \cdot I_i = \frac{3^{b+b_{i-1}}}{2^{a+a_{i-1}}},\]
which is impossible in the $b_i$-note scale since the power of 3 in the above numerator is
\[b+b_{i-1} \geq b_i.\]
This verifies the base case.

Now, let $k \geq 2$, and suppose that whenever
\[b \geq b_i - (k-1)b_{i-1},\]
the note $\frac{3^b}{2^a}$ in the $b_i$-note Pythagorean scale is among the $k-1$ notes preceding a step of size $J_i$.  We aim to prove that, whenever
\[b \geq b_i - kb_{i-1},\]
the note $\frac{3^b}{2^a}$ in the $b_i$-note Pythagorean scale is among the $k$ notes preceding a step of size $J_i$. The only range in which this does not follow immediately from our inductive assumption is
\[b_i - kb_{i-1} \leq b  < b_i - (k-1)b_{i-1},\]
so we restrict to considering $b$ in this range.  For such $b$, the step after $\frac{3^b}{2^a}$ must have size $I_i$, since if not, the following note would be
\[\frac{3^b}{2^a} \cdot J_i = \frac{3^{b+b_{i-1} - b_i}}{2^{a+a_{i-1}-a_i}},\]
which is impossible because the power of 3 in the above numerator is
\[b+b_{i-1} - b_i  < -(k-2)b_{i-1} \leq 0.\]
This means that the note after $\frac{3^b}{2^a}$ in the $b_i$-note Pythagorean scale is
\[\frac{3^b}{2^a} \cdot I_i = \frac{3^{b+b_{i-1}}}{2^{a+a_{i_1}}},\]
in which the power of 3 in the numerator is
\[b+b_{i-1} \geq b_i - (k-1)b_{i-1}.\]
Thus, by our inductive hypothesis, the note after $\frac{3^b}{2^a}$ is among the $k-1$ notes preceding a step of size $J_i$, which means that $\frac{3^b}{2^a}$ itself is among the $k$ notes preceding a step of size $J_i$, as desired.

We have now proven that each of the notes deleted from the $b_i$-note Pythagorean scale to produce the $(b_i - kb_{i-1})$-note Pythagorean scale is among the $k$ notes preceding a step of size $J_i$.  Since there are precisely $kb_{i-1}$ steps of size $J_i$ by Theorem~\ref{thm:old}, this shows that the notes deleted from the $b_i$-note Pythagorean scale to produce the $(b_i - kb_{i-1})$-note Pythagorean scale are precisely the groups of $k$ notes preceding a step of size $J_i$.  Therefore---in the case that the $b_i$-note Pythagorean scale has Type A---the $(b_i - kb_{i-1})$-note Pythagorean scale has the 2-step property, with the two step sizes being $I_i$ and $I_i^kJ_i$.

All of this reasoning carries over essentially unchanged in the case when the $b_i$-note Pythagorean scale has Type B, in which case the key point is that whenever $b \geq b_i - kb_{i-1}$, the note $\frac{3^b}{2^a}$ in the $b_i$-note Pythagorean scale is among the $k$ notes following a step of size $J_i^{-1}$.  Thus, the $(b_i - kb_{i-1})$-note scale, which is obtained from the $b_i$-note scale by deleting these notes, has the 2-step property with step sizes $I_i^{-1}$ and $J_i^{-1}I_i^{-k}$.
\end{proof}

\section{Proof of Main Theorem}
\label{sec:proof2}

Proposition~\ref{prop:deletion} proves one direction of Theorem~\ref{thm:main}.  To see this, note that by \eqref{eq:recursion}, the scale sizes in Proposition~\ref{prop:deletion} are
\begin{equation}
\label{eq:semisrevised}
b_i - kb_{i-1} = b_{i-2} + (k_i-k)b_{i-1}
\end{equation}
with
\[1 \leq k_i - k \leq k_i.\]
Comparing this to the definition in equation \eqref{eq:semiconvergents}, one sees that these are precisely the denominators of semi-convergents of $\log_2(3)$, so the Pythagorean scales whose sizes are these denominators all have the 2-step property.

We now turn to the reverse implication, completing the proof of Theorem~\ref{thm:main}.

\begin{proof}[Proof of Theorem~\ref{thm:main}]
Let $n \geq 2$ be an integer for which the $n$-note Pythagorean scale has the $2$-step property.  Because the sequence of integers $b_i$ is increasing (as one can see, for instance, from \eqref{eq:recursion}), there exists some integer $i$ for which
\[b_{i-1} \leq n < b_i.\]
Our goal is to prove that $n$ is the denominator of a semi-convergent of $\log_2(3)$, which, by the re-writing \eqref{eq:semisrevised}, is equivalent to proving that
\[n = b_i - kb_{i-1}\]
for some integer $1 \leq k \leq k_i -1$.  Suppose, by way of contradiction, that this is not the case.  Then we can choose $k$ such that 
\begin{equation}
    \label{eq:contradiction}
    b_i - kb_{i-1} < n < b_i - (k-1)b_{i-1}.
\end{equation}

Similarly to the proof of Proposition~\ref{prop:deletion}, the $n$-note Pythagorean scale can be obtained from the $b_i$-note Pythagorean scale by deleting the notes $\frac{3^b}{2^a}$ with $b \geq n$.  The key observation is that, if the $b_i$-note scale is organized into $b_{i-1}$ blocks as in Theorem~\ref{thm:old}, then the notes with numerators
\[3^{b_i-1}, \; 3^{b_i - 2}, \; 3^{b_i -3}, \ldots, \; 3^{b_i - b_{i-1}}\]
consist of exactly one note from each block; indeed, as we saw in the proof of the base case of Proposition~\ref{prop:deletion}, these are the last notes in each block in Type A, or the first notes in each block in Type B.  Similarly, the notes with numerators 
\[3^{b_i-b_{i-1}-1}, \; 3^{b_i - b_{i-1} -2}, \; 3^{b_i -b_{i-1}-3}, \ldots, \; 3^{b_i - 2b_{i-1}}\]
consist of exactly one note from each block---namely (as we saw in the proof of the induction step of Proposition~\ref{prop:deletion}), they are the second-to-last notes in each block in Type A, or the second notes in each block in Type B.  Continuing in this way, one finds that the notes with numerators
\begin{equation}
    \label{eq:laststep}
    3^{b_i-(k-1)b_{i-1}-1}, \; 3^{b_i - (k-1)b_{i-1} -2}, \; 3^{b_i -(k-1)b_{i-1}-3}, \ldots, \; 3^{b_i - kb_{i-1}}
\end{equation}
consist of exactly one note from each block for any $1 \leq k \leq k_i -1$.  For instance, when $i=4$, we have
\[b_i = 12, \;\;\; b_{i-1} = 5, \;\;\; k_i = 2\]
by \eqref{eq:tableofks}, so we must have $k=1$ in \eqref{eq:laststep}, and the resulting numerators are
\[3^{11}, 3^{10}, 3^9, 3^8, 3^7,\]
which are indeed the numerators of the first note in each block of \eqref{eq:P12steps}.  Figure~\ref{figure:P41} illustrates the same phenomenon for $i=5$, for which $b_i =41$.

\begin{figure}[ht]
    \begin{tikzpicture}
        % make dots big
        \tikzstyle{dot}=[circle, fill, inner sep=0pt, minimum size=0.3cm]

        % Left box
        \draw[thick] (0,-0.7) -- (4,-0.7);
        \draw[thick] (4,-0.7) -- (4, -0.5);
        \draw[thick] (4,0.5) -- (4,1.75);
        \draw[thick] (4,1.75) -- (0,1.75);
        \draw[thick] (0,1.75) -- (0,0.5);
        \draw[thick] (0,-0.5) -- (0,-0.7);

        % Draw notes
        \node (A1) at (0,0) {$\displaystyle\frac{3^0}{2^0}$};
        \node (A2) at (1,0) {$\displaystyle\frac{3^{12}}{2^{19}}$};
        \node (A3) at (2,0) {$\displaystyle\frac{3^{24}}{2^{38}}$};
        \node (A4) at (3,0) {$\displaystyle\frac{3^{36}}{2^{57}}$};

        % Draw curvy arrows (clockwise from pi to 0)
        \draw[->, thick, >=stealth, bend left=75] (A1) to node[above] {\( I_5 \)} (A2);
        \draw[->, thick, >=stealth, bend left=75] (A2) to node[above] {\( I_5 \)} (A3);
        \draw[->, thick, >=stealth, bend left=75] (A3) to node[above] {\( I_5 \)} (A4);
        \draw[->, thick, >=stealth, bend left=75] (A3) to node[above] {\( I_5 \)} (A4);

        \begin{scope}[shift={(3,0)}]
            \draw[thick] (0,-0.7) -- (4,-0.7);
            \draw[thick] (4,0.5) -- (4, 1.75);
            \draw[thick] (4,-0.7) -- (4, -0.5);
            \draw[thick] (4,1.75) -- (0,1.75);

        % Draw notes
        \node (B1) at (1,0) {$\displaystyle\frac{3^7}{2^{11}}$};
        \node (B2) at (2,0) {$\displaystyle\frac{3^{19}}{2^{30}}$};
        \node (B3) at (3,0) {$\displaystyle\frac{3^{31}}{2^{49}}$};

         % Draw curvy arrows 
         \draw[->, thick, >=stealth, bend left=75] (A4) to node[above] {\( J_5 \)} (B1);
        \draw[->, thick, >=stealth, bend left=75] (B1) to node[above] {\( I_5 \)} (B2);
        \draw[->, thick, >=stealth, bend left=75] (B2) to node[above] {\( I_5 \)} (B3);

        \end{scope}

        \begin{scope}[shift={(6,0)}]
             \draw[thick] (0,-0.7) -- (5.5,-0.7);
             \draw[thick] (5,0.5) -- (5, 1.75);
             \draw[thick] (5,-0.7) -- (5, -0.5);
             \draw[thick] (5.5,1.75) -- (0,1.75);

        % Draw notes
        \node (C1) at (1,0) {$\displaystyle\frac{3^{2}}{2^{3}}$};
        \node (C2) at (2,0) {$\displaystyle\frac{3^{14}}{2^{22}}$};
        \node (C3) at (3,0) {$\displaystyle\frac{3^{26}}{2^{41}}$};
        \node (C4) at (4,0) {$\displaystyle\frac{3^{38}}{2^{60}}$};
        \node (C5) at (5,0) {$\displaystyle\frac{3^{9}}{2^{14}}$};
        \node at (5.75,0) {$\cdots$};

         % Draw curvy arrows 
        \draw[->, thick, >=stealth, bend left=75] (B3) to node[above] {\( J_5 \)} (C1);
        \draw[->, thick, >=stealth, bend left=75] (C1) to node[above] {\( I_5 \)} (C2);
        \draw[->, thick, >=stealth, bend left=75] (C2) to node[above] {\( I_5 \)} (C3);
        \draw[->, thick, >=stealth, bend left=75] (C3) to node[above] {\( I_5 \)} (C4);
        \draw[->, thick, >=stealth, bend left=75] (C4) to node[above] {\( J_5 \)} (C5);

        \end{scope}

    \end{tikzpicture}
    \caption{The beginning of the 41-note scale.  Here, note that $b_5 = 41$, $b_4=12$, and $k_5=3$, so the numerators as in \eqref{eq:laststep} with $k=1$ are $3^{40}, 3^{39}, \ldots, 3^{29}$; these are the numerators of the last note in each block.  Similarly, the numerators as in \eqref{eq:laststep} with $k=2$ are $3^{28}, 3^{27}, \ldots, 3^{17}$, which are the second-to-last note in each block.}
    \label{figure:P41}
\end{figure}
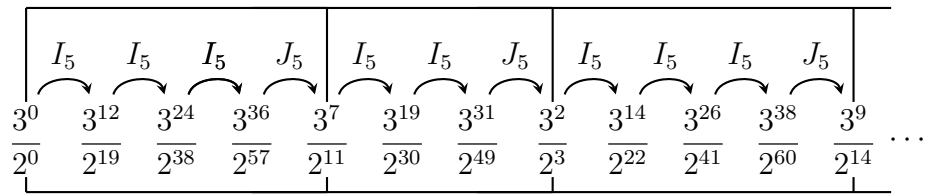

By \eqref{eq:contradiction}, the process of deleting the notes $\frac{3^b}{2^a}$ with $b \geq n$ terminates at some point in the middle of deleting the  notes with numerators as in \eqref{eq:laststep}, so some blocks have undergone $k$ deletions while others have undergone only $k-1$.  Thus, the resulting $n$-note Pythagorean scale has steps of three sizes: $I_n$, $I_n^{k-1}J_n$, and $I_n^kJ_n$ in Type A, or $I_n^{-1}$, $J_n^{-1}I_n^{-(k-1)}$, and $J_n^{-1}I_n^{-k}$ in Type B.  This contradicts the assumption that the scale has the 2-step property, so we have completed the proof of Theorem~\ref{thm:main}.
\end{proof}

\bibliographystyle{plain}
\bibliography{2StepBib}

@article {Clader,
    AUTHOR = {Clader, Emily},
     TITLE = {Why twelve tones? {T}he mathematics of musical tuning},
   JOURNAL = {Math. Intelligencer},
  FJOURNAL = {The Mathematical Intelligencer},
    VOLUME = {40},
      YEAR = {2018},
    NUMBER = {3},
     PAGES = {32--36},
      ISSN = {0343-6993,1866-7414},
   MRCLASS = {00A65 (97M80)},
  MRNUMBER = {3851071},
MRREVIEWER = {Yen\ Duong},
       DOI = {10.1007/s00283-017-9759-1},
       URL = {https://doi.org/10.1007/s00283-017-9759-1},
}

@book {Khinchin,
    AUTHOR = {Khinchin, A. Ya.},
     TITLE = {Continued fractions},
 PUBLISHER = {Dover Publications, Inc., Mineola, NY},
      YEAR = {1997},
     PAGES = {xii+95},
      ISBN = {0-486-69630-8},
   MRCLASS = {11A55 (01A75 11-03 11J70)},
  MRNUMBER = {1451873},
}

@article{CC1989,
  title={Aspects of well-formed scales},
  author={Carey, Norman and Clampitt, David},
  journal={Music Theory Spectrum},
  volume={11},
  number={2},
  pages={187--206},
  year={1989},
  publisher={University of California Press}
}

@article{CC2012,
  title={Two theorems concerning rational approximations},
  author={Carey, Norman and Clampitt, David},
  journal={Journal of Mathematics and Music},
  volume={6},
  number={1},
  pages={61--66},
  year={2012},
  publisher={Taylor \& Francis}
}

@misc{Wilson,
  title={Letter to {C}halmers Pertaining to Moments Of Symmetry},
  author={Wilson, Erv},
  journal={Wilson Archives},
  year={1975},
  note={Available online: https://www.anaphoria.com/mos.pdf}
}

@article{Hellegouarch,
  title={Gammes naturelles},
  author={Hellegouarch, Yves},
  journal={Gazette des math{\'e}maticiens},
  volume={81},
  pages={25--39},
  year={1999}
}

@article{Cubarsi,
  title={An alternative approach to generalized {P}ythagorean scales. {G}eneration and properties derived in the frequency domain},
  author={Cubarsi, Rafael},
  journal={Journal of Mathematics and Music},
  volume={14},
  number={3},
  pages={266--291},
  year={2020},
  publisher={Taylor \& Francis}
}

\end{document}